\documentclass[a4paper, 
headsepline=off, DIV=12, titlepage=false, 11pt]{scrartcl}
\setkomafont{author}{\sffamily\scshape}

\title{
\vspace{-1cm}\semiLARGE
Annihilating class groups in $p$-elementary extensions
}
\date{}
\author{Dominik Bullach \and Daniel Mac{\'i}as Castillo}
 
 \usepackage[backend=bibtex, style=alphabetic, sorting=nyt, maxbibnames=9
 ]{biblatex}
\renewbibmacro{in:}{%
  \ifentrytype{article}{}{\printtext{\bibstring{in}\intitlepunct}}}
  \DeclareFieldFormat{journaltitle}{#1\isdot}
  \DeclareFieldFormat[article, inproceedings, unpublished]{title}{\textit{#1}}
 
\input{preamble}

\begin{document}

\maketitle

\let\thefootnote\relax

\begin{abstract}\vspace{-1cm}
    We derive new cases of conjectures of Rubin and of Burns--Kurihara--Sano concerning derivatives of Dirichlet $L$-series at $s = 0$ in $p$-elementary extensions of number fields for arbitrary prime numbers $p$.  
    In naturally arising examples of such extensions one therefore obtains annihilators of class groups from $S$-truncated Dirichlet $L$-series for `large-enough' sets of places $S$.  
\end{abstract}


\section{Introduction and statement of main results}\label{1}

Stark's Conjecture predicts a description for the leading term of a general Artin $L$-series at $s = 0$ up to an unspecified rational factor. 
Formulating an integral refinement of this conjecture turned out to be a delicate task that Stark himself, in \cite{Stark4}, only found a solution to in the case that the order of vanishing of the $L$-series at $s = 0$ is one. Initial generalisations to higher orders of vanishing, for example the `question' of Stark in \cite{Tangedal, grant} or a conjecture of Sands \cite[Conj.~2.0]{Sands87}, were subsequently shown to not hold in general by Rubin \cite[\S\,4]{Rub96} and Popescu \cite{Popescu07}. Instead, Rubin proposed what is now commonly referred to as the `Rubin--Stark Conjecture' in \textit{loc.\@ cit}.\\  
To state the Rubin--Stark Conjecture, we fix a finite abelian extension of number fields $K / k$ with Galois group $G \coloneqq \gal{K}{k}$ and, following Rubin \cite[Hyp.~2.1]{Rub96}, a triple $(S, V, T)$ of the following form:
\begin{enumerate}[label=(H\arabic*)]
\item $S$ is a finite set of places of $k$ which contains both  the set $S_\infty$ of infinite places of $k$ and the set of places of $k$ that ramify in $K$, 
\item $V \subsetneq S$ is a proper subset comprising places which split completely in $K / k$,
\item $T$ is a finite set of places of $k$ which is disjoint from $S$ and such that the $(S_K, T_K)$-unit group $\mathcal{O}^\times_{K, S, T} \coloneqq \{ a \in K^\times \mid \ord_w (a) = 0 \text{ if } w \not \in S_K, \ord_w (a - 1) > 0 \text{ if } w \in T_K \}$ is $\Z$-torsion free. (Here $S_K$ and $T_K$ denote the sets of places of $K$ that lie above those in $S$ and $T$, respectively, and $\ord_w$ is the normalised valuation attached to $w$.)
\end{enumerate}
We shall refer to such a triple $(S, V, T)$ as a `Rubin datum' for $K / k$. For any Rubin datum $(S, V, T)$ and complex-valued character $\chi$ in $\widehat{G} \coloneqq \Hom_\Z (G, \C^\times)$, the ($S$-truncated, $T$-modified) Dirichlet $L$-series 
\[
L_{k, S, T} (\chi, s) \coloneqq \prod_{v \in T} (1 - \chi (\Frob_v)\NN v^{1 - s} ) \cdot \prod_{v \not \in S} (1 - \chi (\Frob_v) \mathrm{N}v^{-s})^{-1} 
\quad \text{ if } \mathrm{Re}(s) > 1
\]
is well-known to admit a meromorphic continuation to $\C$ that is holomorphic and of order of vanishing at least $|V|$ at $s = 0$ (cf.\@ \cite[Ch.\@ I, Prop.\@ 3.4]{Tate}). We may therefore define the ($|V|$-th order) `Stickelberger element' 
\[
\theta_{K / k, S, T}^{(|V|)} (0) \coloneqq \sum_{\chi \in \widehat{G}} \big( \lim_{s \to 0} s^{-|V|} L_{k, S, T} (\chi^{-1}, s) \big) \cdot e_{\chi}
\]
with $e_\chi \coloneqq |G|^{-1} \sum_{\sigma \in G} \chi (\sigma)^{-1} \sigma$ the usual primitive orthogonal idempotent in $\C [G]$ associated with $\chi$. 
In addition, we define $X_{K, S} \subseteq Y_{K, S} \coloneqq \bigoplus_{w \in S_K} \Z w$ to be the $\Z [G]$-submodule of elements whose coefficients sum to zero, and denote the Dirichlet regulator isomorphism by
\begin{equation} \label{dirichlet regulator}
\lambda_{K, S} \: \R \otimes_\Z \cO_{K, S}^\times \stackrel{\simeq}{\longrightarrow} \R \otimes_\Z X_{K, S},
\quad 
x \otimes a \mapsto - x \sum_{w \in S_K} \log | a|_w \cdot w.
\end{equation} 

The Rubin--Stark Conjecture \cite[Conj.~B']{Rub96} now predicts, via the reinterpretation given in Lemma \ref{alternative definition RS image lem} below, that for every homomorphism of $\ZZ[G]$-modules $f \: \cO_{K, S, T}^\times \to X_{K, S}$ one has 
\[
\theta_{K / k, S, T}^{(|V|)} (0) \cdot {\det}_{\R [G]} (f_\RR\circ \lambda_{K, S}^{-1}) 
\in \Z [G].
\]
Here $f_\RR$ abbreviates the map from $\R \otimes_\Z \cO_{K, S}^\times=\R \otimes_\Z \cO_{K, S,T}^\times$ to $\RR\otimes_\ZZ X_{K,S}$ that is induced by $f$ via extension of scalars to $\R$.

In this note we shall prove a refinement of the above inclusion, and hence also new cases of the Rubin--Stark Conjecture, in certain situations in which Stark's rational conjecture is known due to a result of Tate \cite[Ch.~II, Thm.~6.8]{Tate}. This refinement is most conveniently stated in terms of the `integral dual Selmer group' $\Sel_{K, S, T}$ that is defined
by Burns, Kurihara, and Sano \cite[Def.\@ 2.1]{BKS} as the cokernel of the canonical map 
\[
 \prod_{w \not \in S_K \cup T_K} \Z \to \Hom_\Z ( K_T^\times, \Z ), 
\quad (x_w)_w \mapsto \big \{ a \mapsto \sum_{w} x_w \ord_w (a) \big \}
\]
with $K_T^\times \coloneqq \{ a \in K^\times \mid \ord_w (a - 1) > 0 \text{ if } w \in T_K\}$, and which fits into a canonical exact sequence of $G$-modules
\begin{cdiagram}
    0 \arrow{r} & \Hom_\Z ( \Cl_{K, S, T}, \Q / \Z) \arrow{r} & \Sel_{K, S, T} \arrow{r} & \Hom_\Z ( \cO_{K, S, T}^\times, \Z ) \arrow{r} & 0.
\end{cdiagram}%
Here $\Cl_{K, S, T}$ denotes the Pontryagin dual of the $S_K$-ray class group $\Cl_{K, S, T}$ of $K$ mod $T_K$, defined as the quotient of the group of fractional ideals of $\bigO_{K, S}$ coprime to $T_K$ by the subgroup of principal ideals with a generator congruent to 1 modulo all $w \in T_K$, and all duals are endowed with the contragredient $G$-action.\\
We can now state the first main result of this note, in which we write $\Fitt^n_{\Z [G]} (M)$ for the $n$-th Fitting ideal of a finitely presented $\Z [G]$-module $M$ (for more details on Fitting ideals see, for example, \cite[\S\,3.1]{northcott} or \cite{NickelFitting}) and, given a subset $I$ of $\C [G]$, we denote by $I^\#$ the image of $I$ under the involution of $\C [G]$ defined by sending each $\sigma \in G$ to $\sigma^{-1}$. 

\begin{thm} \label{result 1}
Let $K / k$ be an extension of number fields of one of the following forms: 
\begin{enumerate}[label=(\roman*)]
    \item There exists a prime-power $q$ and a subfield $\kappa$ of $k$ such that $K / \kappa$ is a Galois extension with Galois group isomorphic to the group $\Aff (q)$ of affine transformations of the field $\mathbb{F}_q$ with $q$ elements, and $G = \gal{K}{k}$ is the unique subgroup of order $q$ of $\gal{K}{\kappa}$. 
    \item $K / k$ is a biquadratic extension.
\end{enumerate}
For any Rubin datum $(S, V, T)$ for $K / k$ with $|S| > |V| + 1$ one then has the equality
\begin{equation} \label{image RS}
   \big \{ 
    \theta_{K / k, S, T}^{(|V|)} (0) \cdot {\det}_{\R [G]} ( f_\RR\circ \lambda_{K, S}^{-1})  
    \mid f \in \Hom_{\Z [G]} ( \cO_{K, S, T}^\times, X_{K, S}) \big \} 
    = \Fitt^{|V|}_{\Z [G]} ( \Sel_{K, S, T})^\#,
\end{equation}
as conjectured by Burns, Kurihara, and Sano in \cite[Conj.\@ 7.3]{BKS}. 
\end{thm}

The proof of this result crucially relies on the analytic class number formula, and will be carried out in \S\,\ref{proof of result 1 section}.\\
We remark that Johnston and Nickel \cite[Thm.\@ 7.6]{JoNi16} have previously studied a conjecture of Burns (from \cite{Bur11}) regarding the annihilation of class groups in extensions $K / \kappa$ as in (i) above if $k / \Q$ is abelian.  

\begin{bsp} \label{example affp extension}
Fix a prime number $p$ and let $\zeta_p$ be a primitive $p$-th root of unity in an algebraic closure of $\Q$. Let $\kappa$ be a number field with the property that $\kappa \cap \Q (\zeta_p) = \Q$. If we pick any element $a \in \kappa^\times$ that is not a $p$-th power in $\kappa$, then it is also not a $p$-th power in $k \coloneqq \kappa (\zeta_p)$ and $K \coloneqq k ( \sqrt[p]{a})$ is an extension of the form (i) with $q = p$.  
\end{bsp}

To state our second main result, we fix a prime number $p$ and
write $\Omega$ for the set of subextensions $L/k$ of $K/k$ that have degree equal to $p$.

\begin{thm} \label{result 2}
Let $K / k$ be a $p$-elementary extension of number fields and
    fix a Rubin datum $(S, V, T)$ for $K / k$ that satisfies 
\[
|S| \geq \max \{ |V| + 2, |V| - s_p + (p - 1)(m - 1) + 3\},
\]
where $s_p \coloneqq \dim_{\mathbb{F}_p} (\Cl_{k, S, T} \otimes_\Z \mathbb{F}_p)$ denotes the $p$-rank of the $S$-ray class group mod $T$ of $k$. 

If the equality (\ref{image RS}) holds for all extensions $L / k$ in $\Omega$, then one has that 
\[\big \{ 
    \theta_{K / k, S, T}^{(|V|)} (0) \cdot {\det}_{\R [G]} ( f_\RR\circ \lambda_{K, S}^{-1} )  
    \mid f \in \Hom_{\Z [G]} ( \cO_{K, S, T}^\times, X_{K, S}) \big \}\,\,\subseteq\,\, \Ann_{\Z [G]} (\Cl_{K, S, T}).\]
\end{thm}

To prove Theorem \ref{result 2}, we first show in Lemma \ref{main-result-2-RS-part} that, under the stated hypotheses, the Rubin--Stark Conjecture for $K / k$ is implied by the validity of (\ref{image RS}) for all degree-$p$ subfields. The annihilation statement in Theorem \ref{result 2} is then deduced from this by varying the Rubin datum in combination with Cebotarev's Density Theorem, as in the theory of `Stark systems' (see, for example, \cite[\S\,4]{bss}). Although this latter aspect of the argument is of a general nature, we prefer to focus on the concrete situation at hand in this note and to discuss the general formalism elsewhere.

\begin{rk}
\begin{liste}
\item If $p = 2$, then Theorem \ref{result 2} is unconditional and recovers results of Sands \cite[Thm.\@ 2.2]{Sands} on the Rubin--Stark Conjecture and of Sands \cite[Main Thm.]{Sands12} and the second author \cite[Thm.~1.4]{MaciasCastillo12} on the annihilation of class groups. 
\item If, in the situation of Theorem \ref{result 2}, the stronger bound 
    \[
|S| \geq \max \{ |V| + 2, |V| + (p - 1)(m - 1) + 2\},
    \]
is valid, then the proof of  Theorem \ref{result 2} shows that one has the finer inclusion
\[
\im (\varepsilon^V_{K / k, S, T}) \subseteq \Char_{\Z [G]} ( \Cl_{K, V, T})
  \]
with $\Char_{\Z [G]} ( \Cl_{K, V, T})$ the `characteristic ideal' of $\Cl_{K, V, T}$ defined by Greither--Sakamoto (see \cite[\S\,5.2]{Greither04}, \cite[App.\@ C]{Sakamoto20}). To make this a little more explicit, we note that one both has an inclusion $\Char_{\Z [G]} ( \Cl_{K, V, T}) \subseteq \Ann_{\Z [G]} ( \Cl_{K, V, T})$ and, for every prime number $l \neq p$, an identification 
\[
\Char_{\Z [G]} ( \Cl_{K, V, T}) \otimes_{\Z} \Z_l = \Fitt^0_{\Z [G]} (\Cl_{K, V, T}) \otimes_{\Z} \Z_l.
\]
\end{liste}
\end{rk}

Note that results on the Rubin--Stark Conjecture in the literature outside the classical cases where at most one archimedean place of $k$ splits in $K$ or the degree $[K : k]$ is at most two are extremely sparse (see Remark \ref{known cases rubin--stark} for a full list of known cases).
By combining Theorems \ref{result 1} and \ref{result 2} with Example \ref{example affp extension}, we now obtain the following method to systematically produce new examples in which the conjecture is valid.   

\begin{cor}
Let $p$ be a prime number, $\zeta_p$ a primitive $p$-th root of unity, and $\kappa$ a number field with the property that $\kappa \cap \Q (\mu_p) = \Q$. Let $a_1, \dots, a_m$ be elements of $\kappa$ that are 
$\mathbb{F}_p$-linearly independent in $\kappa^\times / (\kappa^\times)^p$, and set $k \coloneqq \kappa (\mu_p)$ and $K \coloneqq k (\sqrt[p]{a_1}, \dots, \sqrt[p]{a_m})$. If $(S, V, T)$ is a Rubin datum for $K / k$ with 
\[
|S| \geq \max \{ |V| + 2, |V| - s_p + (p - 1)(m - 1) + 3\},
\]
then for every $\Z [G]$-module homomorphism $f \: \cO_{K, S, T}^\times \to X_{K, S}$ one has that 
\[
\theta_{K / k, S, T}^{(|V|)} (0) \cdot {\det}_{\R [G]} (f_\RR\circ \lambda_{K, S}^{-1})
\in \Ann_{\Z [G]} (\Cl_{K, S, T}).
\]
In particular, the Rubin--Stark Conjecture holds for $(S, V, T)$ and $K / k$. 
\end{cor}

\begin{proof}
    The kernel of the natural map $\kappa^\times / (\kappa^\times)^p \to k^\times / (k^\times)^p$ identifies with $H^1 (\gal{K}{\kappa}, \mu_p)$, and hence vanishes. It follows that $a_1, \dots, a_m$ generate an $\mathbb{F}_p$-subvectorspace of $k^\times / (k^\times)^p$ of dimension $m$. By Kummer theory, one therefore has that $[K : k] = p^m$ and so, noting that $\gal{K}{\kappa} \cong \mathrm{Aff}(p)$ because $\kappa \cap \Q (\mu_p) = \Q$, the result follows by combining Theorems \ref{result 1} and \ref{result 2}. 
\end{proof}

\paragraph{Acknowledgements}
The first author wishes to acknowledge the financial support of the Engineering and Physical Sciences Research Council [EP/W522429/1].
The second author acknowledges support for this article as part of Grants CEX2019-000904-S, PID2019-108936GB-C21 and PID2022-142024NB-I00 funded by MCIN/AEI/ 10.13039/501100011033.

\section{Rubin--Stark elements}

Let $(S, V, T)$ be a Rubin datum for the finite abelian extension of number fields $K / k$ and
fix a labelling $S = \{ v_0, \dots, v_{|S| - 1} \}$ such that $V = \{ v_1, \dots, v_{|V|} \}$ along with an extension $w_i$ of each place $v_i$ in $S$. 
The `Rubin--Stark element' $\varepsilon^V_{K / k, S, T}$ for $(S, V, T)$ is then the unique element of $\R \otimes_\Z \exprod^{|V|}_{\Z [G]} \cO^\times_{K, S}$ with the property that
\[
\big( \exprod^{|V|} \lambda_{K, S} \big) (\varepsilon^V_{K / k, S, T}) = \theta^{(|V|)}_{K / k, S, T} (0) \cdot \bigwedge_{1 \leq i \leq |V|} ( w_i - w_0) 
\]
with $\exprod^{|V|} \lambda_{K, S} \: \R \otimes_\Z \exprod^{|V|}_{\Z [G]} \cO^\times_{K, S} \stackrel{\simeq}{\longrightarrow} \R \otimes_\Z \exprod^{|V|}_{\Z [G]} X_{K, S}$ the isomorphism induced by (\ref{dirichlet regulator}).\\
We then define the `image' of $\varepsilon^V_{K / k, S, T}$ to be the $\Z [G]$-submodule of $\R [G]$ given by
\[
\im ( \varepsilon^V_{K / k, S, T} ) \coloneqq \big \{ F (\varepsilon^V_{K / k, S, T}) \mid F \in \exprod^{|V|}_{\Z [G]} \Hom_{\Z [G]}( \cO^\times_{K, S, T}, \Z [G]) \big \},
\]
where we have written $F (\varepsilon^V_{K / k, S, T})$ for the image of $(\varepsilon^V_{K / k, S, T}, F)$ under the determinant pairing
\begin{align*}
\big( \R \otimes_\Z \exprod^{|V|}_{\Z [G]} \cO^\times_{K, S}  \big)
\times 
\big( \R \otimes_\Z \exprod^{|V|}_{\Z [G]} \Hom_{\Z [G]}( \cO^\times_{K, S, T}, \Z [G]) \big) 
& \to \R [G], \\
(a_1 \wedge \dots \wedge a_{|V|}, f_1 \wedge \dots \wedge f_{|V|} ) 
& \mapsto \det ( f_i ( a_j))_{1 \leq i, j \leq |V|}.
\end{align*}

The following result allows us to reformulate the equality (\ref{image RS}) in terms of Rubin--Stark elements.  

\begin{lem} \label{alternative definition RS image lem}
For any Rubin datum $(S, V, T)$ for $K / k$, one has an equality 
\[
\im ( \varepsilon^V_{K / k, S, T} ) = \{ {\det}_{\R [G]} ( f_\RR \circ \lambda_{K, S}^{-1}) \cdot \theta^{|V|}_{K / k, S, T} (0) \mid f \in \Hom_{\Z [G]} ( \cO_{K, S, T}^\times, X_{K, S}) \}.
\]  
\end{lem}

\begin{proof}
This is an immediate consequence of \cite[Lem.~2.2]{MaciasCastillo12}. 
\end{proof}

\begin{rk} \label{known cases rubin--stark}
The `Rubin--Stark Conjecture' \cite[Conj.~B']{Rub96} predicts that $\im (\varepsilon^V_{K / k, S, T})$ is contained in $\Z [G]$ for any Rubin datum $(S, V, T)$. To the best of the authors' knowledge, the following is a complete list of cases in which the Rubin--Stark Conjecture is known at present. 
\begin{liste}
\item If $K = k$, then the conjecture is a direct consequence of the analytic class number formula. 
\item The Rubin--Stark Conjecture holds if $k = \Q$. If $V = S_\infty$ is the singleton comprising the unique infinite place of $\Q$, then this follows by a direct computation that shows that $\varepsilon^V_{K / k, S, T}$ can be expressed in terms of a cyclotomic unit (cf.\@ \cite[Ch.\@ III, \S\,5]{Tate}).
The general case, even more generally for $k$ a finite abelian extension of $\Q$, is a consequence of the `equivariant Tamagawa Number Conjecture' (cf.\@ \cite[Thm\@ 3.1\,(i)]{Burns07} or \cite[Thm.\@ 5.12]{BKS}) which is known to hold for finite abelian extensions of $\Q$ by work of Burns and Greither \cite{BurnsGreither} and Flach \cite{Flach11}. 
\item If $k$ is an imaginary quadratic field and $V = S_\infty$ is the singleton comprising the unique infinite place of $k$, then the Rubin--Stark Conjecture follows from Kronecker's Second Limit Formula for elliptic units (cf.\@ \cite[Ch.\@ IV, Prop.\@ 3.9]{Tate}). In addition, the conjecture is known in general for extensions $K / k$, with $k$ a finite abelian extension of an imaginary quadratic field, for which the equivariant Tamagawa Number Conjecture is known to hold. In this direction, the reader is referred to recent work of Hofer and the first author \cite[Thm.\@ B]{BullachHofer}.
\item If $V = \emptyset$, then the conjecture is a consequence of work of Cassou-Nogu\`es \cite{Cassou} and, independently, Deligne and Ribet \cite{DeligneRibet} (cf.\@ \cite[Prop.\@ 3.7]{gross88}).
\item If $k$ is a totally real field and $K$ is CM, then the conjecture follows from work of Dasgupta and Kakde \cite{DasguptaKakde} on the Strong Brumer--Stark Conjecture, up to an unspecified power of $2$ (cf.\@ \cite[Thm.\@ 1.6]{DasguptaKakde}). Moreover, Dasgupta, Kakde, Silliman and Wang have recently announced a full proof of the Rubin-Stark Conjecture in this setting.
\item Rubin has proved \cite[Thm.\@ 3.5]{Rub96} that the Rubin--Stark Conjecture holds if $[K : k] = 2$.
\item Grant \cite{grant} has verified the Rubin--Stark Conjecture for $k = \Q ( \zeta_5)$ and $K = k ( \sqrt[5]{\epsilon})$ with $\zeta$ a primitive $5$-th root of unity and $\epsilon \coloneqq - \zeta^2 - \zeta^3$.
\item If $K / k$ is a multi-quadratic extension, then Dummit, Sands, and Tangedahl \cite{DummitSandsTangedahl}, Sands \cite{Sands}, and the second author \cite{MaciasCastillo12} have verified the conjecture in special cases. 
 \item The conjecture holds if $S \setminus V$ contains a place that splits completely in $K$ (cf.\@ \cite[Prop.\@ 3.1]{Rub96}).    
 \item McGown, Sands, and Valli\`eres \cite{McGown--Sands--Vallieres} have numerically verified the conjecture for the 19197 examples of cubic extensions $K / k$ with $K$ a totally real field of discriminant less than $10^{12}$, $k$ a real quadratic field, and $V = S_\infty$. 
\end{liste}
\end{rk}

\section{The proof of Theorem \ref{result 1}} \label{proof of result 1 section}

\subsection{Weil-\'etale cohomology complexes}

In this preliminary section, we briefly recall key properties of a useful family of complexes constructed by Burns, Kurihara, and Sano in \cite{BKS}. To do so, we let $K/F$ be an arbitrary finite Galois extension of number fields with group $\Delta_F\coloneqq\gal{K}{F}$.\\
We write $D(\ZZ[\Delta_F])$ for the derived category of $\ZZ[\Delta_F]$-modules and $D^\mathrm{p}(\ZZ[\Delta_F])$ for its full triangulated subcategory comprising complexes that are `perfect', that is, isomorphic (in $D(\ZZ[\Delta_F])$) to a bounded complex of finitely generated projective $\ZZ[\Delta_F]$-modules.

\begin{lemma}\label{WeLemma} Fix sets $S$ and $T$ of places of $F$ that satisfy the conditions (H1) and (H3) in \S\,\ref{1} with $k$ replaced by $F$.
Then the `Weil-\`etale cohomology complex'
\[
C^\bullet_{K, S, T} \coloneqq \text{R} \Hom_\Z ( \text{R} \Gamma_{c, T} ( ( \bigO_{K, S})_\mathcal{W}, \Z), \Z) [-2]
\] 
constructed in \cite[Prop.~2.4]{BKS} is an object of $D^{\mathrm{p}}(\ZZ[\Delta_F])$ that has the following properties.\begin{itemize}
\item[(i)] $C^\bullet_{K,S,T}$ is acyclic outside degrees zero and one, with $H^0(C^\bullet_{K,S,T})=\cO_{K,S,T}^\times$, and the `transpose Selmer group' $\Sel^\mathrm{tr}_{K,S,T} \coloneqq H^1(C^\bullet_{K,S,T})$ lies in a short exact sequence of $\Delta_F$-modules
\begin{cdiagram}
    0\arrow{r} & \Cl_{K,S,T}\arrow{r} & \Sel^\mathrm{tr}_{K,S,T}\arrow{r} &  X_{K,S}\arrow{r} &  0.
\end{cdiagram}
\item[(ii)] $C^\bullet_{K,S,T}$ is isomorphic in $D(\ZZ[\Delta_F])$ to a complex $[P_0 \stackrel{\phi}{\to} P_1]$ in which $P_0$ is finitely generated projective (and placed in degree 0) while $P_1$ is free of finite rank.
\item[(iii)] For any normal subgroup $\Gamma$ of $\Delta_F$ there is, in $D^{\mathrm{p}}(\ZZ[\Delta_F/\Gamma])$, a canonical isomorphism 
\[
\ZZ[\Delta_F/\Gamma]\otimes^{\mathbb{L}}_{\ZZ[\Delta_F]}C^\bullet_{K,S,T}\cong C^\bullet_{K^\Gamma,S,T}.
\]
\end{itemize}
\end{lemma}

\begin{proof}
$C^\bullet_{K,S,T}$ is an object of $D^\mathrm{p}(\ZZ[\Delta_F])$ by choice of $S$ and by \cite[Prop.~2.4\,(iv)]{BKS}. Claim (i) is Remark 2.7 in loc.\@ cit. Claim (ii) is proved in \S\,5.4 of loc. cit. Claim (iii) follows from the diagram in Prop.\@ 2.4\,(i) of loc.\@ cit.\@ and the functoriality properties of \'etale cohomology.
\end{proof}

\subsection{The proof in case (i)}
In this subsection we assume the hypotheses of Theorem \ref{result 1}\,(i). In particular, $\Delta \coloneqq \gal{K}{\kappa}$ is isomorphic to $\Aff (q)$, and $(S, V, T)$ is a Rubin datum for $K / k$ with $|S| > |V| + 1$.\\
Since $G  = \gal{K}{k}$ is abelian, the complex $C^\bullet_{K,S,T}$ in $D^{\mathrm{p}}(\ZZ[G])$ admits a well-defined determinant $\Det_{\Z [G]} (C^\bullet_{K, S, T})$ (in the sense of Knudsen--Mumford).
We then also use the `zeta element' $z_{K / k, S, T} \in \R \otimes_\Z \Det_{\Z [G]} ( C^\bullet_{K, S, T})$, the definition of which can be found in \cite[Def.~3.5]{BKS} and will be recalled in the course of the proof of Lemma \ref{lemma 1} below. For the moment we only note that $z_{K / k, S, T}$ is by construction an $\R [G]$-basis of the free rank-one $\R [G]$-module $\R \otimes_\Z \Det_{\Z [G]} ( C^\bullet_{K, S, T})$. 

\begin{lem} \label{lemma 1}
The following claims are valid. 
\begin{liste}
\item The zeta element $z_{K / k, S, T}$ belongs to $\Q \otimes_\Z \Det_{\Z [G]} ( C^\bullet_{K, S, T})$. In particular, $z_{K / k, S, T}$ is a $\Q [G]$-basis of the free rank-one $\Q [G]$-module  $\Q \otimes_\Z \Det_{\Z [G]} ( C^\bullet_{K, S, T})$. 
\item For every prime number $\ell$, there exists an element $\fz^{(\ell)}_{K / k, S, T}$ of $\Det_{\Z [G]} ( C^\bullet_{K, S, T})$ with the following properties: 
\begin{enumerate}[label=(\roman*)]
    \item The $\Z [G]$-submodule of $\Det_{\Z [G]} ( C^\bullet_{K, S, T})$ generated by $\fz^{(\ell)}_{K / k, S, T}$ has prime-to-$\ell$ index. 
    \item The unique element $\lambda^{(\ell)} \in \Q [G]$ defined by $z_{K / k, S, T} = \lambda^{(\ell)} \cdot  \fz^{(\ell)}_{K / k, S, T}$ belongs to the image of the map
\[
\rho_{\Delta / G} \: \zeta(\C [\Delta]) \to \C [G], \quad x \mapsto \sum_{\chi \in \widehat{G}} \big( \prod_{\psi \in \widehat{\Delta}} \psi (x)^{\langle \psi, \Ind_G^\Delta( \chi) \rangle} \big) \cdot e_\chi,
\]
where $\widehat{\Delta}$ is the set of irreducible characters of $\Delta$, $\langle \cdot, \cdot \rangle$ denotes the inner product of characters, $\zeta(\C [\Delta]) \cong \prod_{\psi \in \widehat{\Delta}} \C$ denotes the centre of $\C [\Delta]$, and we have written $\psi$ for the map $\zeta(\C [\Delta]) \to \C$ induced by $\psi$.
\end{enumerate}
\end{liste}

\end{lem}

\begin{proof}
Claim (a) is equivalent to Stark's Conjecture for $K / k$ (cf.\@ \cite[Thm.\@ 7.1\,b)]{Flach04}).
Since any non-trivial character of $G$ induces a rational-valued character of $\Delta$, the validity of Stark's Conjecture follows from Tate's proof of Stark's Conjecture for rational-valued characters in \cite[Ch.~II, Thm.~6.8]{Tate}. \\
To prove claim (b), we may enlarge $S$ and $T$ since, if $S'$ and $T'$ are respective disjoint finite oversets of $S$ and $T$, then the exact triangles in \cite[Prop.\@ 2.4, (ii) and right hand column of (6) in (i)]{BKS} induce an isomorphism
\[
\Det_{\Z [G]} (C^\bullet_{K, S', T'}) \stackrel{\simeq}{\longrightarrow} \Det_{\Z [G]} (C^\bullet_{K, S, T})
\]
that maps $z_{K / k, S', T'}$ to $z_{K / k, S, T}$. We therefore may and will
assume that $S$ contains all places that are ramified in $K / \kappa$ and that both $S$ and $T$ are stable under the action of $\Delta$. 

Since the complex $C^\bullet_{K,S,T}$ depends only on $K$, $S_K$ and $T_K$, we may then regard it also as an object of $D^{\mathrm{p}}(\ZZ[\Delta])$. We fix a representative of $C^\bullet_{K,S,T}$ in $D(\ZZ[\Delta])$ as in Lemma \ref{WeLemma}\,(ii) (applied to $F = \kappa$). We note that (\ref{dirichlet regulator}) combines with the Noether--Deuring Theorem to imply that $\Q \otimes_\Z P_0 \cong \Q \otimes_\Z P_1$. For every prime number $\ell$, Roiter's Lemma \cite[(31.6)]{CurtisReiner} then gives the existence of an injection $i^{(\ell)} \: P_1 \hookrightarrow P_0$ with finite cokernel of order prime to $\ell$.

We fix a set $\{ \sigma_1, \dots, \sigma_{(\Delta : G)}\}$ of representatives for $\Delta / G$ and choose an ordered $\Z [\Delta]$-basis $\mathfrak{B} = \{ b_1, \dots, b_d \}$ of $P_1$. 
Then $P_1$ is also a free $\Z [G]$-module, with (ordered) $\Z [G]$-basis 
\[ 
\mathfrak{B}' \coloneqq \{ \sigma_1 b_1, \dots \sigma_{(\Delta : G)} b_1, \dots, 
\sigma_1 b_d, \dots \sigma_{(\Delta : G)} b_d\}. 
\]
We also define ordered sets $\mathfrak{C}^{(\ell)} \coloneqq \{ i^{(\ell)} (b) \mid b \in \mathfrak{B} \}$ and ${\mathfrak{C}'}^{(\ell)} = \{ i^{(\ell)} (b) \mid b \in \mathfrak{B}'\}$.
Setting $P_1^\ast \coloneqq \Hom_{\Z [G]} (P_1, \Z [G])$, we now define
\[
\fz^{(\ell)}_{K / k, S, T} \coloneqq \big( \bigwedge_{c \in {\mathfrak{C}'}^{(\ell)}} c \big)
\otimes \big( \bigwedge_{b \in \mathfrak{B}'} b^\ast \big)
\; \in \; \big( \exprod_{\Z [G]}^{(\Delta : G)d} P_0 \big) \otimes_{\Z [G]} \big( \exprod_{\Z [G]}^{(\Delta : G)d} P_1^\ast \big) = \Det_{\Z [G]} (C^\bullet_{K, S, T}),
\]
where $b^\ast \: P_1 \to \Z [G]$ denotes the $\Z [G]$-linear dual of $b \in P_1$. By construction, the element $\fz^{(\ell)}_{K / k, S, T}$ then has property (i). \\ 
To justify claim (ii), we first recall the definition of the zeta element $z_{K / k, S, T}$. Our fixed choice of representative for $C^\bullet_{K, S, T}$ gives rise to exact sequences $0 \to \cO_{K,S, T}^\times \to P_0 \to \phi (P_0) \to 0$ and $\phi (P_0) \to P_1 \to \Sel^\mathrm{tr}_{K,S,T} \to 0$ of $\Z [\Delta]$-modules for which we may choose $\R [\Delta]$-splittings 
\[
\iota_1 \: \R \otimes_\Z P_0 \cong (\R \otimes_\Z \cO_{K,S, T}^\times)
\oplus ( \R \otimes_\Z \phi (P_0)),
\quad 
\iota_2 \: \R \otimes_\Z P_1 \cong (\R \otimes_\Z X_{K,S}) 
\oplus ( \R \otimes_\Z \phi (P_0)).
\]
Given this, we define the composite isomorphism of $\R [\Delta]$-modules
\[
\alpha \coloneqq (\iota_2^{-1} \circ (\lambda_{K, S} \oplus \id ) \circ \iota_1) \: P_0 \to P_1,
\]
where $\lambda_{K, S}$ denotes the Dirichlet regulator map defined in (\ref{dirichlet regulator}). We write $A^{(\ell)}$ for the matrix in $\mathrm{GL}_{(\Delta : G) d} ( \R [G])$ that represents $\alpha$ with respect to the bases ${\mathfrak{C}'}^{(\ell)}$ and $\mathfrak{B}'$. 

We consider the `leading term' 
\[
\theta^\ast_{K / \kappa, S, T} (0) \coloneqq 
\sum_{\psi \in \widehat{\Delta}} L^\ast_{\kappa, S, T} (\check\psi, 0) e_{\psi}\in\zeta(\RR[\Delta])^\times,
\]
where $\check\psi$ denotes the contragredient of $\psi$ and $L^\ast_{\kappa, S, T} (\check\psi, 0)$ is the leading term of $L_{\kappa, S, T} (\check\psi, s)$ at $s = 0$. Similarly, we set $\theta^\ast_{K / k, S, T} (0) \coloneqq 
\sum_{\chi \in \widehat{G}} L^\ast_{k, S, T} (\check\chi, 0) e_{\chi}\in\RR[G]^\times$.
One then has that $z_{K / k, S, T} = \lambda^{(\ell)} \cdot \fz^{(\ell)}_{K / k, S, T}$ with $\lambda^{(\ell)} \in \R [G]^\times$ the unique element such that
$
\lambda^{(\ell)} \cdot {\det}_{\R [G]} ( A^{(\ell)}) = \theta^\ast_{K / k, S, T} (0) $.
The reduced norm of the matrix $B^{(\ell)} \in \mathrm{GL}_d (\R [\Delta])$ that represents $\alpha$ with respect to the bases $\mathfrak{C}^{(\ell)}$ and $\mathfrak{B}$ belongs to $\zeta(\RR[\Delta])^\times$, and we define a scalar $\mu^{(\ell)} \in \zeta(\R [\Delta])^\times$ by
\[
\mu^{(\ell)} \cdot \mathrm{Nrd}_{\R [\Delta]} ( B^{(\ell)}) = \theta^\ast_{K / \kappa, S, T}(0). 
\]
By the functoriality of reduced norms under restriction to subgroups (see, for example, \cite[bottom of p.\@ 291]{Breuning04}) one has $\rho_{\Delta/G}(\mathrm{Nrd}_{\R [\Delta]} ( B^{(\ell)}))={\det}_{\R [G]} ( A^{(\ell)})$ and thus also
$$\rho_{\Delta / G} ( \mu^{(\ell)}) \cdot {\det}_{\R [G]} ( A^{(\ell)})=\rho_{\Delta / G} ( \theta^\ast_{K / \kappa, S, T}(0))=\theta^\ast_{K / k, S, T}(0),$$
from which we deduce that $\rho_{\Delta / G} ( \mu^{(\ell)}) = \lambda^{(\ell)}$. 
This concludes the proof of claim (b).
\end{proof}

We now give the proof of Theorem \ref{result 1} in case (i). 
 Since $\Delta \cong \Aff (q)$, one has that $\widehat{\Delta}$ consists of the linear characters of $\Delta / G$ and the unique irreducible character of degree $q - 1$ that is obtained as $\psi_\mathrm{nl} \coloneqq \Ind_G^\Delta (\chi)$ for any non-trivial character $\chi$ of $G$ (see, for example, \cite[Thm.\@ 5]{Motose}). As a consequence, one has
\[
\langle \psi, \Ind_G^\Delta( \chi) \rangle = 
\begin{cases}
1 \quad & \text{ if } \chi \neq \bm{1}_G, \psi = \psi_\mathrm{nl}, \\
1 & \text{ if } \chi = \bm{1}_G, \psi = \bm{1}_\Delta, \\
0 & \text{ otherwise}. 
\end{cases}
\]
For every prime number $\ell$, the element $\lambda^{(\ell)}$ provided by Lemma \ref{lemma 1}\,(b)\,(ii) is hence of the form $\lambda^{(\ell)} = a e_{\bm{1}} + b (1 - e_{\bm{1}})$ for suitable $a, b \in \Q$. \\
Now, the isomorphism $\ZZ\otimes^\mathbb{L}_{\Z [G]}C^\bullet_{K, S, T} \cong C^\bullet_{k, S, T}$ in Lemma \ref{WeLemma}\,(iii) induces an isomorphism $\ZZ\otimes_{\ZZ[G]}( \Q \otimes_\Z \Det_{\Z [G]} ( C^\bullet_{K, S, T})) \cong \Q \otimes_\Z \Det_\Z (C^\bullet_{k, S, T})$ that sends $1 \otimes z_{K / k, S, T}$ to $z_{k/k, S, T}$. In addition, the analytic class number formula for $k$ asserts that $z_{k/k, S, T}$ is a $\Z$-basis of the free rank-one $\Z$-module $\Det_\Z (C^\bullet_{k, S, T})$.

For each prime number $\ell$, we write $\ZZ_{(\ell)}$ for the localisation of $\ZZ$ at  the prime ideal $\ell\Z$.
The definition of $\fz_{K  / k, S, T}^{(\ell)}$ then implies that both $1 \otimes \fz_{K  / k, S, T}^{(\ell)}$ and $a \cdot (1 \otimes \fz_{K / k, S, T}^{(\ell)}) = 1 \otimes ( \lambda^{(\ell)} \fz_{K / k, S, T}^{(\ell)})  =  1 \otimes z_{K / k, S, T}$ are $\ZZ_{(\ell)}$-bases of $\ZZ\otimes_{\ZZ[G]}( \ZZ_{(\ell)} \otimes_\Z \Det_{\Z [G]} ( C^\bullet_{K, S, T})) $. We conclude that $a$ belongs to $\ZZ_{(\ell)}^\times$. \\
We next write $\NN = \NN_{\Q [G] / \Q} \: \Q [G] \to \Q$ for the ring-theoretic norm map and note that the construction of \cite[Lem.\@ 3.7\,(c)]{scarcity} gives the existence of an $\NN$-semilinear map 
$\mathcal{F} \: \Q \otimes_\Z \Det_{\Z [G]} ( C^\bullet_{K, S, T}) \to \Q \otimes_\Z \Det_\Z (C^\bullet_{K, S, T})$ that sends $z_{K / k, S, T}$ to $z_{K / K, S, T}$. Since $z_{K / K, S, T}$ is a $\Z$-basis of $\Det_\Z (C^\bullet_{K, S, T})$ by the analytic class number formula for $K$, we see that for each prime $\ell$, both $\mathcal{F} ( \fz_{K / k, S, T}^{(\ell)})$ and $z_{K / k, S, T} = \mathcal{F} ( z_{K / k, S, T}) = \mathcal{F} ( \lambda^{(\ell)} \fz_{K / k, S, T}^{(\ell)}) = \NN ( \lambda^{(\ell)}) \cdot \mathcal{F} ( \fz_{K / k, S, T}^{(\ell)})$ are $\Z_{(\ell)}$-bases of $\Z_{(\ell)} \otimes_\Z \Det_\Z (C^\bullet_{K, S, T})$. It follows that $\NN (\lambda^{(\ell)}) = a b^{q - 1}$ must also belong to $\Z_{(\ell)}^\times$. Upon recalling that $a \in \Z_{(\ell)}^\times$ by the above discussion, we conclude that $b^{q - 1} \in \Z_{(\ell)}^\times$. Since $b$ is rational, we deduce that $b$ belongs to $\Z_{(\ell)}^\times$.  \\
Define an idempotent $e_{K, S, V}$ of $\Q [G]$ as the sum $\sum_\chi e_\chi$ of all primitive orthogonal idempotents $e_\chi$ associated with characters $\chi$ of $G$ such that $e_\chi$ annihilates $\CC\otimes_\ZZ X_{K, S \setminus V}$. \\
We then define a `projection map' $\Theta_{K / k, S}^V$ as the composite map
\begin{align*}
 \Q \otimes_\Z \Det_{\Z [G]} ( C^\bullet_{K, S}) 
& \xrightarrow{\phantom{\cdot e_{K, S, V}}} \Det_{\Q [G]} ( \Q \otimes_\Z \cO_{K, S}^\times) 
\otimes_{\Q [G]} \Det_{\Q [G]} ( \Q \otimes_\Z X_{K, S})^{-1} \\
& \xrightarrow{\cdot e_{K, S, V}}
e_{K, S, V} \cdot \big ( ( \Q \otimes_\Z \exprod^{|V|}_{\Z [G]} \cO_{K, S, T}^\times)
\otimes_{\Q [G]} ( \Q \otimes_\Z \exprod^{|V|}_{\Z [G]} Y_{K, V} )^{-1} \big) \\
& \xrightarrow{\quad \simeq \quad} e_{K, S, V} \cdot ( \Q \otimes_\Z \exprod^{|V|}_{\Z [G]} \cO_{K, S, T}^\times),
\end{align*}
where the first arrow is the natural `passage-to-cohomology' map, the second map is induced by multiplication by $e_{K, S, V}$, and the last arrow by the trivialisation $\exprod^{|V|}_{\Z [G]} Y_{K, V} \cong \Z [G]$ that is afforded by sending $\exprod_{1 \leq i \leq |V|} w_i$ to $1$. \\
Note that our hypothesis $|S| > |V| + 1$ combines with the short exact sequence $0\to X_{K,S\setminus V}\to X_{K,S}\to Y_{K,V}\to 0$ to imply that $e_{\bm{1}}\cdot e_{K, S, V}=0$. In particular, we have $\lambda^{(\ell)} \cdot e_{K, S, V} = (a e_{\bm{1}} + b ( 1 - e_{\bm{1}})) \cdot e_{K, S, V} = b e_{K, S, V}$. 
Since it is proved in \cite[Thm.\@ 5.14]{BKS} that one has $\Theta_{K / k, S}^V ( z_{K / k, S, T}) = \varepsilon^V_{K / k, S, T}$, we therefore deduce that
\[
\varepsilon^V_{K / k, S, T} = \Theta_{K / k, S, T}^V ( z_{K / k, S, T}) 
= \lambda^{(\ell)} \cdot \Theta_{K / k, S, T}^V (\fz^{(\ell)}_{K / k, S, T}) = b \cdot \Theta_{K / k, S, T}^V (\fz^{(\ell)}_{K / k, S, T})
\]
for each prime $\ell$. Now, the equality $\Z_{(\ell)} \otimes_\Z \im ( \Theta_{K / k, S, T}^V (\fz^{(\ell)}_{K / k, S, T}))
= \Z_{(\ell)} \otimes_\Z \Fitt^{|V|}_{\Z [G]} ( \Sel_{K, S, T})^\#$ that is established via the argument of
\cite[Thm.\@ 7.5]{BKS} combines with the last displayed equation and the fact that $b$ is invertible to imply that 
\[
\Z_{(\ell)} \otimes_\Z \im ( \varepsilon^V_{K / k, S, T}) 
= \Z_{(\ell)} \otimes_\Z  \big ( b \cdot \im ( \Theta_{K / k, S, T}^V (\fz^{(\ell)}_{K / k, S, T})) \big)
= \Z_{(\ell)} \otimes_\Z \Fitt^{|V|}_{\Z [G]} ( \Sel_{K, S, T})^\#.
\]
The claim in Theorem \ref{result 1}\,(i) now follows upon recalling that $\ell$ is an arbitrary prime number. 

\subsection{The proof in case (ii)}

To prove Theorem \ref{result 1} in case (ii), we let $K/k$ be a biquadratic extension of number fields and note that, by the 
known validity of Stark's Conjecture for $K / k$, the zeta element $z_{K / k, S, T}$ is a $\Q [G]$-basis of the free rank-one $\Q [G]$-module  $\Q \otimes_\Z \Det_{\Z [G]} ( C^\bullet_{K, S, T})$ (cf.\@ the argument of Lemma \ref{lemma 1}\,(a)). 
We then let $\ell$ be an arbitrary prime number and choose, using Roiter's Lemma, an element $\fz_{K / k, S, T}^{(\ell)}$ that generates a $\Z [G]$-submodule of $\Det_{\Z [G]} ( C^\bullet_{K, S, T})$ of finite, prime-to-$\ell$ index. Label the proper intermediate fields of $K / k$ as $K_1 \coloneqq k, K_2, K_3$, and $K_4$, and, using Lemma \ref{WeLemma}(iii), denote the image of $\fz_{K / k, S, T}^{(\ell)}$ under the natural map 
\[
\Det_{\Z [G]} ( C^\bullet_{K, S, T}) \to 
\Z [\gal{K_i}{k}]\otimes_{\Z [G]}\Det_{\Z [G]} ( C^\bullet_{K, S, T})   \cong 
\Det_{\Z [\gal{K_i}{k}]} ( C^\bullet_{K_i, S, T})
\]
as $\fz_{K_i / k, S, T}^{(\ell)}$ for every $i \in \{1, \dots, 4 \}$. 
Write $\chi_i$ for the trivial character if $i = 1$ and the non-trivial character of $\gal{K_i}{k}$ otherwise.
The discussion above (in case (i)) then shows that we have 
\[
e_{\chi_i} \cdot \fz_{K_i / k, S, T}^{(\ell)} = a_i \cdot e_{\chi_i} \cdot z_{K_i / k, S, T}
\]
for some $a_i$ in $\Z_{(\ell)}^\times$. It follows that
\[
\fz_{K / k, S, T}^{(\ell)} = 
( \sum_{i = 1}^4 a_i e_{\chi_i}) \cdot z_{K / k, S, T}. 
\]
If $\ell \neq 2$, then it is clear that $\lambda^{(\ell)} \coloneqq \sum_{i = 1}^4 a_i e_{\chi_i}$ belongs to $\Z_{(\ell)} [G]^\times$. For $\ell = 2$, the scalar $\lambda^{(2)}$ belongs to $\Z_{(2)} [G]^\times$ if and only if it belongs to $\Z_{(2)} [G]$ because $\NN_{\Q [G] / \Q} (\lambda^{(2)}) = \prod_{i = 1}^4 a_i$ is a unit in $\Z_{(2)}$. Now, $\lambda^{(2)}$ is in $\Z_{(2)} [G]$ if and only if, for every $\sigma \in G$ we have that
\[
\sum_{i = 1}^4  a_i \chi_i (\sigma) \equiv 0 \mod 4. 
\]
Note that $\chi_i (\sigma) = \pm 1$ and $a_i \equiv \pm 1 \mod 4$ for all $i \in \{1, \dots, 4\}$. One can then check explicitly that the above congruence holds if and only if $\prod_{i = 1}^4 a_i \equiv 1 \mod 4$ (cf.\@ also \cite[Lem.~6.3\,(v)]{Buckingham}).  
In particular, if we let $b \in \{ \pm 1 \}$ be defined by $b \equiv \prod_{i = 1}^4 a_i \mod 4$, then $\lambda' \coloneqq b a_1 e_{\bm{1}} + \sum_{i = 1}^3 a_i e_{\chi_i}$ belongs to $\Z_{(2)} [G]^\times$. \\
Furthermore, the assumption $|S| > |V| + 1$ ensures that $\lambda e_{K, S, V} = \lambda' e_{K, S, V}$ and so we obtain 
\[
\Z_{(2)} \otimes_\Z \im ( \varepsilon^V_{K / k, S, T}) 
= \Z_{(2)} \otimes_\Z \lambda' \cdot \im ( \Theta_{K / k, S, T}^V (\fz^{(2)}_{K / k, S, T}))
= \Z_{(2)} \otimes_\Z \Fitt^{|V|}_{\Z [G]} ( \Sel_{K, S, T})^\#.
\]
As the corresponding identity holds for each odd $\ell$, this completes the proof of Theorem \ref{result 1}. \qed 
\begin{rk}
The only instances of (i) and (ii) in Theorem \ref{result 1} that can neither be treated by the argument used to prove Theorem \ref{result 1} nor Remark \ref{known cases rubin--stark}\,(i) are the cases in which $|S| = |V| + 1$ and the unique place $v \in S \setminus V$ has full decomposition group in $K / k$.
In any such situation and for `large-enough' $V$, the inclusion (\ref{image RS}) is in fact equivalent to the relevant case of the `equivariant Tamagawa Number Conjecture' and amounts to a subtle question about signs. To make this more explicit, we suppose that $K / k$ is biquadratic, $|S| = |V| + 1$, and $V$ is large enough such that $\Cl_{K, S, T}$ vanishes. 
Then $\cO_{K, S, T}^\times$ is a free $\Z [G]$-module of rank $|V|$ and we can choose an ordered $\Z [G]$-basis $\mathfrak{B}$ of $\cO_{K, S, T}^\times$. Fix an ordering $G  = \{ g_1, g_2, g_3, g_4 \}$, and define an ordered $\Z$-basis $\mathfrak{B}'$ of $\cO_{K, S, T}^\times$ by setting $\mathfrak{B}' \coloneqq \{ g b \mid g \in G, b \in \mathfrak{B}\}$ ordered lexicographically. Similarly, we set $\mathfrak{W} \coloneqq \{ g w_i \mid g \in G, 1 \leq i \leq |V| \}$ ordered lexicographically. 
Then one can show that (\ref{image RS}) is equivalent to
\[
{\det}_\R (  \log | b |_w )_{b \in \mathfrak{B}', w \in \mathfrak{W}} < 0.
\]
(Cf.\@ \cite[Prop.\@ 10.5]{Buckingham}.) This question does not depend on the ordering on $G$ and, since 
$G$ is $\Z / 2 \Z \oplus \Z / 2 \Z$, also not on the choice of basis $\mathfrak{B}$ (or the ordering on it) because every unit in $\Z [G]$ is of the form $\pm g$ for some $g \in G$ in this case, and so has norm 1. 
\end{rk}

\section{The proof of Theorem \ref{result 2}} \label{proof of result 2 section}

We now fix a $p$-elementary extension $K/k$ with $G \cong (\nZ{p})^m$. Write $\Omega^\ast$ for the set of subgroups $H$ of $G$ of index at most $p$.
The following algebraic observation plays a key role in the sequel. 

\begin{lem} \label{algebraic-lemma}
Set $\NN_H = \sum_{\tau \in H} \tau$ for every $H \in \Omega^\ast$.
In $\Z [G]$ we then have the equality
\[
\sum_{H \in \Omega^\ast} \NN_H + \Big ( (p^{m - 1}  - 1) - \Big (\sum_{i = 0}^{m - 1} p^i \Big)\Big) \cdot \NN_G = p^{m - 1}.
\]
\end{lem}

\begin{proof}
Observe that $G$ is an $\mathbb{F}_p$-vector space and the (non-trivial) $H$ are exactly the $(m - 1)$-dimensional subspaces of $G$. Recall that the trace pairing
\[
\mathbb{F}_p^m \times \mathbb{F}_p^m \to \mathbb{F}_p, \quad (v, w) \mapsto \sum_{i = 1}^m v_i w_i
\]
is perfect, hence induces a bijection between $(m - 1)$-dimensional and 1-dimensional subspaces. 
The number of 1-dimensional subspaces is exactly $\frac{p^{m} - 1}{p - 1}$, 
hence $| \Omega^\ast \setminus \{ G \} |$ is equal to $ \frac{p^{m} - 1}{p - 1}$.
If we fix $v \in \mathbb{F}_p \setminus \{ 0 \}$, then the set of all $(m - 1)$-dimensional subspaces of $\mathbb{F}_p^m$ that contain $v$ is in bijection with all 1-dimensional subspaces of the space 
$\{ w \in \mathbb{F}_p^m \mid \sum_{i = 1}^m v_i w_i = 0 \}$,
the kernel of the $(1 \times m)$-matrix $v$. This space is therefore of dimension $m - 1$ and contains $\frac{p^{m - 1} - 1}{p - 1}$ subspaces of dimension one. 
That is, there are exactly $\frac{p^{m - 1} - 1}{p - 1}$ subgroups $H \in \Omega^\ast \setminus \{ H \}$ that contain a given (non-trivial) element of $G$. It follows that there are exactly 
\begin{align*}
\frac{p^{m} - 1}{p - 1} - \frac{p^{m - 1} - 1}{p - 1} & = \frac{(p^m - 1) - (p^{m - 1} - 1)}{p - 1} = \frac{p^{m - 1} ( p - 1)}{(p - 1)} = p^{m - 1} 
\end{align*}
such $H$ that do \textit{not} contain a given (non-trivial) element. Thus, each element of $G$ appears in the sum $( \sum_{H \in \Omega^\ast \setminus \{ G \}} \NN_H \Big) + p^{m - 1} ( N_G - 1)$
exactly $| \Omega^\ast \setminus \{ G \} |$ many times. From this we get
\begin{align*}
\Big ( \sum_{H \in \Omega^\ast \setminus \{ G \}} \NN_H \Big) + p^{m - 1} ( N_G - 1) &
= | \Omega^\ast \setminus \{ G \} | \cdot \NN_G  = \frac{(p^{m} - 1)}{p - 1} \cdot \NN_G = \Big ( \sum_{i = 0}^{m - 1} p^i \Big) \cdot \NN_G. \qedhere
\end{align*}\end{proof}

For any integer $r \geq 0$ and $H\in\Omega^\ast$, we consider the injection
\[
\nu_H \: \C \otimes_\Z \exprod^r_{\Z [G / H]} \bigO_{K^H, S, T}^\times \to \C \otimes_\Z \exprod^r_{\Z [G]} \bigO_{K, S, T}^\times, \quad a \mapsto |H|^{\mathrm{max}\{0, 1 - r\}} \cdot a
\]
that satisfies
\begin{equation} \label{injection and norms}
\nu_H ( \NN_H^r a) = \NN_H a \quad \text{ for any } a \in \C \otimes_\Z \exprod^r_{\Z [G]}  \bigO_{K, S, T}^\times.
\end{equation}
As a straightforward application of Lemma \ref{algebraic-lemma} we obtain the following consequence
that recovers \cite[Prop.\@ 4.5]{Sands} in the case $p = 2$. 

\begin{prop} \label{prop0} In $\R \otimes_\Z \exprod^r_{\Z [G]} \bigO_{K, S, T}^\times$ we have the equality
\[
\varepsilon^V_{K / k, S, T} = \frac{1}{p^{m - 1}} \cdot \Big (  \sum_{H \in \Omega^\ast} \nu_H \big( \varepsilon_{K^H / k, S, T}^V \big)  + \Big ( (p^{m - 1}  - 1) - \Big (\sum_{i = 0}^{m - 1} p^i \Big) \Big) \cdot \nu_G \big ( \varepsilon_{k / k, S, T}^V \big) \Big).
\]

\end{prop}

\begin{proof}
 Using Lemma \ref{algebraic-lemma}\,(a), equation (\ref{injection and norms}), and the norm relations for Rubin--Stark elements \cite[Prop.~6.1]{Rub96} we calculate
\begin{align*}
p^{m - 1} \cdot \varepsilon_{K / k, S, T}^V & = \Big ( \sum_{H \in \Omega^\ast} \NN_H + \Big ( (p^{m - 1}  - 1) - \Big (\sum_{i = 0}^{m - 1} p^i \Big)\Big) \cdot \NN_G   \Big) \cdot \varepsilon_{K / k, S, T}^V \\
& = \Big ( \sum_{H \in \Omega^\ast} \nu_H \big ( \NN_H^{|V|} \varepsilon_{K / k, S, T}^V\big) \Big) + \Big ( (p^{m - 1}  - 1) - \Big (\sum_{i = 0}^{m - 1} p^i \Big) \Big) \cdot \nu_G \big (\NN_G^{|V|} \varepsilon_{K / k, S, T}^V\big)   \\
& =  \Big ( \sum_{H \in \Omega^\ast} \nu_H \big ( \varepsilon_{K^H / k, S, T}^V \big) \Big) + \Big ( (p^{m - 1}  - 1) - \Big (\sum_{i = 0}^{m - 1} p^i \Big) \Big) \cdot \nu_G \big (\NN_G^{|V|} \varepsilon_{k / k, S, T}^V \big).
\qedhere
\end{align*}
\end{proof}

To prepare for the proof of Theorem \ref{result 2}, we now first give a preliminary result in which we write $I_G \coloneqq \ker \{ \Z [G] \to \Z \}$ for the absolute augmentation ideal of $\Z [G]$ and, given a $\Z [G]$-module $M$ and non-negative integer $r$, define its `$r$-th exterior bidual' to be
\[
\bidual^r_{\Z [G]} M \coloneqq \big \{ a \in \Q \otimes_\Z \exprod^r_{\Z [G]} M \mid F (a) \in \Z [G] \text{ for all } F \in \exprod^r_{\Z [G]} \Hom_{\Z [G]} (M, \Z [G]) \big \}.
\]

\begin{lem} \label{main-result-2-RS-part}
Fix a Rubin datum $(S, V, T)$ for $K / k$ and a non-negative integer $c$ that satisfies 
\[
|S| \geq \max \{ |V| + 2, |V| - s_p + (p - 1)(m - 1) + 2 + c\},
\]
where $s_p \coloneqq \dim_{\mathbb{F}_p} (\Cl_{k, S, T} \otimes_\Z \mathbb{F}_p)$ denotes the $p$-rank of $\Cl_{k, S, T}$. 
If the equality (\ref{image RS}) holds for all extensions $L / k$ in $\Omega$, then $\varepsilon^V_{K / k, S, T}$ belongs to $I_G^c \cdot \bidual^{|V|}_{\Z [G]} \cO_{K, S, T}^\times$. 
\end{lem}

\begin{proof}
At the outset we note that, for any $H \in \Omega^\ast$, the map $\nu_H$ restricts to an injection $\bidual^{|V|}_{\Z [G / H]} \cO_{K^H, S, T}^\times \to \bidual^{|V|}_{\Z [G]} \cO_{K, S, T}^\times$ (cf.\@ \cite[Rk.\@ 4.13]{BKS}). 
By Proposition \ref{prop0}, it is hence sufficient to prove that $\varepsilon^V_{K^H / k, S, T}$ belongs to $p^{m - 1} I_{G / H}^c \bidual^{|V|}_{\Z [G / H]} \cO_{K^H, S, T}^\times$ for every $H \in \Omega^\ast$. By the assumption $|S| \geq |V| + 2$, we may and will assume $K^H \neq k$ so that $K^H \in \Omega$.\\ 
We now first claim that for this purpose it is enough to prove that $\im (\varepsilon^V_{K^H / k, S, T})$ is contained in $p^{m - 1} I_{G / H}^{ 1 + c}$. 
To justify this, we apply Lemma \ref{WeLemma}\,(ii) to fix a representative $[P_0 \stackrel{\phi}{\to} P_1]$ of the complex $C^\bullet_{K^H, S, T}$ in $D^{\mathrm{p}}(\ZZ[G/H])$. From \cite[Lem.\@ B.6]{Sakamoto20} we then obtain an exact sequence
\begin{equation} \label{ryotaro exact sequence}
    \begin{tikzcd}
    0 \arrow{r} & \bidual^{|V|}_{\Z [G / H]} \cO_{K^H, S, T}^\times \arrow{r} & \exprod^{|V|}_{\Z [G / H]} P_0 \arrow{r}{\phi} & P_1 \otimes_{\Z [G / H]} \exprod^{|V| - 1}_{\Z [G/ H]} P_0. 
    \end{tikzcd}
\end{equation}%
In particular, we may view $\varepsilon^{V}_{K^H / k, S, T}$ as an element of $\exprod^{|V|}_{\Z [G / H]} P_0$. Now, if $\im (\varepsilon^V_{K^H / k, S, T})$, which equals $ \{F ( \varepsilon^{V}_{K^H / k, S, T}) \mid F \in \exprod^{|V|}_{\Z [G / H]} \Hom_{\Z [G]}(P_0, \Z [G]) \}$, is contained in $p^{m - 1} I_{G / H}^{1 + c}$, then $\varepsilon^{V}_{K^H / k, S, T}$ belongs to the module $p^{ m - 1}  I_{G / H}^{1 + c} \exprod^{|V|}_{\Z [G]} P_0$ (cf.\@ \cite[Prop.\@ 4.17]{BKS}). We may therefore write $\varepsilon^{V}_{K^H / k, S, T} = p^{m - 1} (\sigma_H - 1)^{ 1 + c} a$ with $\sigma_H$ a generator of $G / H$ and $a$ an element of $\exprod^{|V|}_{\Z [G]} P_0$. From the exact sequence (\ref{ryotaro exact sequence}) we then see that  
\[
p^{m - 1} (\sigma_H - 1)^{1 + c} \cdot \phi (a) = 
\phi ( p^{m - 1} (\sigma_H - 1)^{1 + c} a) = \phi ( \varepsilon^{V}_{K^H / k, S, T}) = 0.
\]
Since $P' \coloneqq P_1 \otimes_{\Z [G / H]} \exprod^{|V| - 1}_{\Z [G/ H]} P_0$ is $\Z$-torsion free, this implies that $(\sigma_H - 1)^{1 + c} \cdot \phi (a)$ vanishes. As $(\sigma_H - 1)P'$ and $(P')^{G / H} = \ker \{ P' \stackrel{\cdot(\sigma_H - 1)}{\longrightarrow} P' \}$ intersect trivially because $P'$ is $G/H$-cohomologically trivial, it then follows by induction on $c$ that $(\sigma_H - 1) \phi (a)$ vanishes. Exactness of (\ref{ryotaro exact sequence}) now shows that $(\sigma_H - 1) a$ belongs to $\bidual^{|V|}_{\Z [G / H]} \cO_{K^H, S, T}^\times$, as required to prove that $\varepsilon^V_{K^H, S, T}$ belongs to $p^{m - 1} (\sigma_H - 1)^c \bidual^{|V|}_{\Z [G / H]} \cO_{K^H, S, T}^\times$.\\ 
It now remains to prove that $\im (\varepsilon^V_{K^H / k, S, T})$ is contained in $p^{m - 1} I_{G / H}^{1 + c}$. We may and will 
assume that no place in $S \setminus V$ splits completely in $K^H / k$, since otherwise $\varepsilon_{K^H / k, S, T}^V$ vanishes. Thus, every place in $S \setminus V$ has full decomposition group in $K^H / k$.
Since we assume (\ref{image RS}) to hold for $K^H / k$ it is enough to prove, in this situation, that $\Fitt^{|V|}_{\Z [G / H]} (\Sel_{K^H, S, T})^\#\subseteq p^{m - 1}I_{G / H}^{1 + c}$. \\
To verify this inclusion, we use the `transpose' Selmer group defined in Lemma \ref{WeLemma}\,(i) and the equality 
\[
\Fitt^{|V|}_{\Z [G / H]}  (\Sel_{K^H, S, T})^\# = \Fitt^{|V|}_{\Z [G / H]} ( \Sel^\mathrm{tr}_{K^H, S, T})
\]
of \cite[Lem.\@ 2.8]{BKS}. It then suffices to verify that $\Fitt^{|V|}_{\Z [G / H]} ( \Sel^\mathrm{tr}_{K^H, S, T})\subseteq p^{m - 1} I_{G / H}^{1+c}$.\\ 
For this purpose, we first note that $Y_{K^H, V}$ is a free direct summand of $X_{K^H,S}\cong Y_{K^H, V} \oplus X_{K^H, S\setminus V}$, hence also of $\Sel^\mathrm{tr}_{K^H, S, T}$. We may thus find a $\Z [G / H]$-module $M$ such that $\Sel^\mathrm{tr}_{K^H, S, T} \cong M \oplus Y_{K^H, V}$ and one has the following modified version of the exact sequence in Lemma \ref{WeLemma}\,(i):
\begin{equation} \label{modified exact sequence}
\begin{tikzcd}
    0 \arrow{r} & \Cl_{K^H,S, T} \arrow{r} & M \arrow{r} & X_{K^H, S\setminus V} \arrow{r} & 0.
\end{tikzcd}
\end{equation}
Setting $d \coloneqq | S \setminus V|$, one has $X_{K^H, S \setminus V} \cong \Z^{d - 1}$ and fixing again a generator $\sigma_H$ of $G/H$,
 \[
 \Fitt^0_{\Z [G / H]} (X_{K^H, S \setminus V} ) = I_{ G / H}^{d - 1} = (\sigma_H - 1)^{d - 1} \Z [G / H].
 \]
In particular, $\Fitt^0_{\Z [G / H]} ( X_{K^H, S \setminus V})$ is a principal ideal and so we may apply \cite[Lem.\@ 2.5\,(ii)]{JoNi13} to the exact sequence (\ref{modified exact sequence})
to infer that
\begin{align*}
\Fitt^{|V|}_{\Z [G / H]} (\Sel^\mathrm{tr}_{K^H, S, T})= \Fitt^{0}_{\Z [G / H]} (M) & = \Fitt^0_{\Z [G/ H]} (  \Cl_{K^H,S, T}) \cdot \Fitt^0_{\Z [G / H]} ( X_{K^H, S \setminus V}) \\
& = \Fitt^0_{\Z [G/ H]} (  \Cl_{K^H,S, T}) \cdot I_{G / H}^{d - 1}. 
\end{align*}
Fix a place $v \in S \setminus V$ and recall that we may assume that $v$ has full decomposition group in $K^H / k$. If we write $H_{S, T} (K^H)$ and $H_{S, T} (k)$ for the $(S, T)$-ray class fields of $K^H$ and $k$, respectively, then $H_{S, T} (k) \cap K^H = k$ since $v$ splits completely in $H_{S, T} (k)$. Thus, we may identify $\gal{H_{S, T} (k)}{k} \cong \gal{K^H \cdot H_{S, T} (k)}{K^H}$ and hence the restriction map $\gal{H_{S, T} (K^H)}{K^H} \to \gal{H_{S, T} (k)}{k}$ is surjective. By class field theory, the restriction map corresponds with the norm map $\Cl_{K^H,S, T} \to \Cl_{k,S, T}$ and so, in particular, the map $\Cl_{K^H,S, T} \to \Cl_{k,S, T} \otimes_\Z \mathbb{F}_p \cong (\nZ{p})^{s_p}$ is surjective as well. This map is $G / H$-equivariant, thus we obtain an inclusion
\[
\Fitt^0_{\Z[G / H]} ( \Cl_{K^H,S, T}) \subseteq \Fitt^0_{\Z [G / H]} \big( (\Z /p \Z)^{s_p} \big) =  \prod_{i = 1}^{s_p} ( p \Z [G / H] + I_{G / H})
\subseteq \sum_{i = 0}^{s_p} p^i I_{G / H}^{s_p - i}.
\]
By the previous discussion, we therefore have an inclusion
\[
\Fitt^{|V|}_{\Z [G / H]} (\Sel^\mathrm{tr}_{K^H, S, T}) \subseteq \big ( \sum_{i = 0}^{s_p} p^i I_{G / H}^{s_p - i} \big) \cdot I_{G / H}^{d - 1} 
= \sum_{i = 0}^{s_p} p^i I_{G / H}^{s_p - i + d - 1}
\subseteq \big ( \sum_{i = 0}^{s_p} p^i I_{G / H}^{s_p - i + d - c - 1} \big) \cdot I_{G / H}^c.
\]
Since $\sigma_H$ is of order $p$, we have $(\sigma_H - 1)^p \equiv \sigma_H^p - 1 = 0 \mod p$ and so $(\sigma_H - 1)^p$ is divisible by $p$ in $\Z [G / H]$. Noting that the quotient $\Z [G / H] / I_{G / H} \cong \Z$ is torsion-free, we see that $(\sigma_H - 1)^p$ is in fact divisible by $p (\sigma_H - 1)$. From this it follows that $(\sigma_H - 1)^{s_p - i + d -c - 1}$ is divisible by $p^{\max \{ 0, \lfloor (s_p - i + d - c - 2) / (p -1) \rfloor\}} (\sigma_H - 1)$. As a consequence, 
\[
\sum_{i = 0}^{s_p} p^i I_{G / H}^{s_p - i + d - c - 1}
\subseteq 
\sum_{i = 0}^{s_p} p^{i + \lfloor (s_p - i + d - c - 2) / (p -1) \rfloor} I_{G / H}
\subseteq p^{\lfloor (s_p + d - c - 2) / (p -1) \rfloor} I_{G / H}, 
\]
where we have used that
\[
i + \lfloor \frac{s_p - i + d - c - 2}{p -1} \rfloor = \lfloor \frac{(p - 1) i + s_p - i + d - c - 2}{p -1} \rfloor \geq \lfloor \frac{s_p + d - c - 2}{p -1} \rfloor
\]
as a consequence of $ p -1 \geq 1$. Now,
\[
\frac{d + s_p - c - 2}{p -1} = 
\frac{|S| - |V| + s_p - c - 2}{p -1} \geq  m - 1
\iff 
|S| \geq  |V| - s_p + (p - 1) ( m - 1) + 2 + c
\]
and so $\Fitt^{|V|}_{\Z [G / H]} (\Sel^\mathrm{tr}_{K^H, S, T})$ is contained in $p^{m - 1} I_{G / H}^{1 + c}$ as soon as $|S| \geq  |V| - s_p + (p - 1) ( m - 1) + 2 + c$.
This concludes the proof that $\im (\varepsilon^V_{K / k, S, T})$ is contained in $p^{m - 1} I_{G / H}^{1 + c}$, as required.
\end{proof}

We can now give the proof of Theorem \ref{result 2}.
\medskip \\
\textit{Proof (of Theorem \ref{result 2}):}
Write $H_{k, p}$ and $H_K$ for the extensions of $k$ and $K$ that correspond with $\Cl_{k, S, T} \otimes_\Z \mathbb{F}_p$ and $\Cl_{K, S, T}$ via class field theory. That is, $H_{k, p}$ is the maximal $p$-elementary extension of $k$ that is unramified outside $T$ and in which all places in $S$ split completely, and $H_K$ is that maximal extension of $K$ that is unramified outside $T_K$ and in which all places in $S_K$ split completely. Note that $H_K$ is Galois over $k$. By Cebotarev's Density Theorem, we may then choose a finite set $W$ of prime ideals of $k$ that has all of the following properties:
\begin{enumerate}[label=(\roman*)]
    \item $W$ is disjoint from $S \cup T$,
    \item every place in $W$ splits completely in $K \cdot H_{k, p}$,
    \item $\{ \Frob_\p \mid \p \in W \}$ is a generating set for $\gal{H_K}{ K \cdot H_{k, p}}$.
\end{enumerate}
In particular, one has $\Cl_{k, S', T} \otimes_\Z \Z_p = \Cl_{k, S, T} \otimes_\Z \mathbb{F}_p$ with $S' \coloneqq S \cup W$.
Class field theory then provides for a commutative diagram
\begin{cdiagram}
\Cl_{K, S', T} \arrow{r}{\simeq} \arrow{d}{\widetilde{\NN}_{K / k}} & \gal{H_{k, p} K}{K} \arrow[hookrightarrow]{d} \\
\Cl_{k, S, T} \otimes_\Z \mathbb{F}_p \arrow{r}{\simeq} & \gal{H_{k, p}}{k}, 
\end{cdiagram}%
where the right hand vertical arrow is the natural restriction map and $\widetilde{\NN}_{K / k}$ is the composite of the `norm' map $\Cl_{K, S', T} \to \Cl_{k, S', T}$ induced by the norm $\NN_{K / k} \: K^\times \to k^\times$ and the projection 
$ \Cl_{k, S', T} \to \Cl_{k, S, T} \otimes_\Z \mathbb{F}_p$. As a consequence, we obtain a $G$-equivariant isomorphism $\Cl_{K, S', T} \cong \widetilde{\NN}_{K / k} ( \Cl_{K, S', T})$,
and hence an exact sequence of $\Z [G]$-modules
\begin{equation} \label{free-resolution}
\begin{tikzcd}[column sep=small]
0 \arrow{r} & \cO_{K, S, T}^\times 
\arrow{r} & \cO_{K, S', T}^\times \arrow{r}{\psi} & Y_{K, W} \arrow{r} & 
\Cl_{K, S, T} \arrow{r} & \widetilde{\NN}_{K / k} ( \Cl_{K, S', T}) \arrow{r} & 0
\end{tikzcd}
\end{equation}
with $\psi \: \cO_{K, S', T}^\times \to Y_{K, W}$ the map  $ a \mapsto \sum_{w \in W_K} \ord_w (a) w$. 
Fix a labelling $W = \{ v_{|S| + 1}, \dots, v_{|S'|}\}$ and, for each $i \in \{ |S| + 1, \dots, |S'|\}$, an extension $w_i$ of $v_i$ to $K$. 
By condition (ii) every place of $K$ above a fixed $w_i$ is of the form $\sigma w_i$ for some $\sigma \in G$, which allows us to define a map $w_i^\ast \: Y_{K, W} \to \Z [G]$ by sending $\sum_{w \in W_K} a_w w$ to $\sum_{\sigma \in G} a_{\sigma w} \sigma$ (so $w_i^\ast$ is the `dual' of $w_i$).  
Setting $A \coloneqq \ker \{ \Cl_{K, S', T} \to \widetilde{\NN}_{K / k} ( \Cl_{k, S, T}) \}$, the exact sequence (\ref{free-resolution}) then implies that
\[
\bigwedge_{i = |S| + 1}^{|S'|} (w_i^\ast \circ \psi) ( \cO_{K, S', T}^\times) \subseteq \Ann_{\Z [G]} (A).
\]
We now claim that $\im (\varepsilon^V_{K / k, S, T})$ is contained in $I_G$ times the intersection on the left hand side. 
To do this, we first note that 
$s'_p \coloneqq \dim_{\mathbb{F}_p} ( \Cl_{k, S', T} \otimes_\Z \mathbb{F}_p)$ is equal to $s_p$ because $\Cl_{k, S', T} \otimes_\Z \mathbb{F}_p = \Cl_{k, S, T} \otimes_\Z \mathbb{F}_p$ by condition (ii). 
Setting $V' \coloneqq V \cup W$, it then follows that \begin{align*}
|S'|  = |W| + |S| 
& \geq |W| + \max \{ |V| + 2, |V| - s_p + (p - 1)(m - 1) + 3 \}\\
& \geq \max \{ |V'| + 2, |V'| - s'_p + (p - 1)(m - 1) + 3 \}.
\end{align*}
By Lemma \ref{main-result-2-RS-part}, we therefore have that 
$\varepsilon^{V'}_{K / k, S', T}$ belongs to $I_G \cdot \bidual^{|V'|}_{\Z [G]} \cO_{K, S', T}^\times$, hence can be written as $\varepsilon^{V'}_{K / k, S', T} = \sum_{i = 1}^t x_i a_i$ with a natural number $t$ and elements $x_1, \dots, x_t \in I_G$ and $a_1, \dots, a_t \in \bidual^{|V'|}_{\Z [G]} \cO_{K, S', T}^\times$. 
Set $\psi_l \coloneqq w_l^\ast \circ \psi$ and, for every
$f \in \exprod^{|V|}_{\Z [G]} (\cO_{K, S, T}^\times)^\ast$, define  a map 
\[
\Phi_{j, f} \coloneqq \R \otimes_\Z \exprod^{|V'|}_{\Z [G]} \cO_{K, S', T}^\times 
\to \R \otimes_\Z \cO_{K, S, T}^\times, \quad 
a \mapsto (f \circ \exprod_{\substack{|S| + 1 \leq l \leq |S'| \\ l \neq j}} \psi_j) ( a). 
\]
For every $g \in (\cO_{K, S, T}^\times)^\ast$ and $i\in \{1, \dots, t\}$, one then has that $(g \circ \Phi_{j, f}) ( a_i)$ belongs to $\Z [G]$. This  shows that
\[
\Phi_{j, f} (a_i) \in
\big \{ a \in \R \otimes_\Z \cO_{K, S', T}^\times \mid g (a) \in\Z [G] 
\text{ for all } g \in (\cO_{K, S, T}^\times)^\ast \} = \cO_{K, S, T}^\times
\]
because $\cO_{K, S, T}^\times$ is $\Z$-torsion free. 
For any $f \in \exprod^{|V|}_{\Z [G]} (\cO_{K, S, T}^\times)$ and $j \in \{ |S| + 1, \dots, |S'| \}$, we obtain that
\begin{align*}
    ( f \circ \big( \exprod_{|S| + 1\leq l \leq |S'|} \psi_l \big) \big) (a_i)
    = \pm (\psi_j \circ \Psi_{j, f}) ( a_i) \subseteq \psi_j ( \cO_{K, S, T}^\times).
\end{align*}
By the above discussion, this shows that $ ( f \circ \big( \exprod_{|S| + 1\leq l \leq |S'|} \psi_l \big) \big) (a_i)$ belongs to $\Ann_{\Z [G]} (A)$ for every $i$.
Now, by \cite[Prop.\@ 3.6]{Sano} (see also \cite[Prop.\@ 5.2]{Rub96}) one has 
\[
\big( \exprod_{|S| + 1\leq l \leq |S'|} \psi_l \big) (\varepsilon^{V'}_{K / k, S', T}) = \pm \varepsilon^V_{K / k, S, T}
\]
and so, for any $f \in \exprod^{|V|}_{\Z [G]} (\cO_{K, S, T}^\times)$, we deduce that
\begin{align*}
    f ( \varepsilon^V_{K / k, S, T})
    & = \pm \big( f \circ \big( \exprod_{|S| + 1\leq l \leq |S'|} \psi_l \big) \big) ( \varepsilon^{V'}_{K / k, S', T}) 
     = \pm \sum_{i = 1}^t x_i \cdot \big( f \circ \big( \exprod_{|S| + 1\leq l \leq |S'|} \psi_l \big) \big) (a_i) \\
    & \subseteq I_G \cdot \Ann_{\Z [G]} (A).
\end{align*}
As $\widetilde{\NN}_{K / k} ( \Cl_{K, S, T})$ (which carries the trivial $G$-action) is annihilated by $I_G$, we conclude from the tautological exact sequence 
\begin{cdiagram}
0 \arrow{r} & A \arrow{r} & \Cl_{K, S, T} \arrow{r} & \widetilde{\NN}_{K / k} ( \Cl_{K, S, T}) \arrow{r} & 0
\end{cdiagram}%
 that any element in $\im (\varepsilon^V_{K / k, S, T})$ annihilates $\Cl_{K, S, T}$, as required to prove Theorem \ref{result 2}.
\qed

\enlargethispage{0.5cm}
\printbibliography

\footnotesize

\textsc{King's College London,
Department of Mathematics,
London WC2R 2LS,
UK} \\
\textit{Email address:} \href{mailto:dominik.bullach@kcl.ac.uk}{dominik.bullach@kcl.ac.uk}
\smallskip\\ 
\textsc{Departamento de Matemáticas, Universidad Autónoma de Madrid, 28049 Madrid, Spain;
and Instituto de Ciencias Matemáticas, 28049 Madrid, Spain}\\
\textit{Email address:} \href{mailto:daniel.macias@uam.es}{daniel.macias@uam.es}

\end{document}